\documentclass{amsart}

\usepackage{enumerate}
\usepackage{amssymb}
\usepackage{latexsym}
\usepackage{longtable}
\usepackage{times}
\usepackage{graphics}

\copyrightinfo{2024}{}

\makeatletter

\newcommand{\ZZ}{\mathbb{Z}}

\renewcommand{\leq}{\leqslant}
\renewcommand{\geq}{\geqslant}

\newcommand{\lcm}{\mathop{\operator@font lcm}}

\newcommand{\diag}{\mathop{\operator@font diag}}
\newcommand{\im}{\mathop{\operator@font im}}
\newcommand{\ord}[1]{\mathop{\operator@font ord}(#1)}
\newcommand{\supp}[1]{\mathop{\operator@font supp}(#1)}

\newcommand{\C}[1]{\mathop{\operator@font C \kern 0pt}_{#1}}
\newcommand{\Sym}[1]{\mathop{\operator@font S \kern 0pt}_{#1}}
\newcommand{\SYM}[1]{\mathop{\operator@font Sym \kern 0pt}(#1)}
\newcommand{\Alt}[1]{\mathop{\operator@font A \kern 0pt}_{#1}}
\newcommand{\ALT}[1]{\mathop{\operator@font Alt \kern 0pt}(#1)}
\newcommand{\GL}[2]{\mathop{\operator@font GL}(#1,#2)}
\newcommand{\SL}[2]{\mathop{\operator@font SL}(#1,#2)}
\newcommand{\PSL}[2]{\mathop{\operator@font PSL}(#1,#2)}
\newcommand{\AFF}[1]{\mathop{\operator@font AFF}(#1)}

\newcommand{\RCWA}[1]{\mathop{\operator@font RCWA}(#1)}
\newcommand{\RCWAp}[1]{\mathop{\operator@font RCWA^+}(#1)}
\newcommand{\CT}[1]{\mathop{\operator@font CT}(#1)}
\newcommand{\CTP}[2]{\mathop{\operator@font CT}_{#1}(#2)}
\newcommand{\CTint}[1]{\mathop{\operator@font CT}_{\rm int}(#1)}
\newcommand{\CS}[1]{\mathop{\operator@font CS}(#1)}
\newcommand{\CSP}[2]{\mathop{\operator@font CS}_{#1}(#2)}
\newcommand{\Mod}[1]{\mathop{\operator@font Mod}(#1)}
\newcommand{\Mult}[1]{\mathop{\operator@font Mult}(#1)}
\newcommand{\Div}[1]{\mathop{\operator@font Div}(#1)}

\makeatother

\allowdisplaybreaks[1]

\theoremstyle{definition} \newtheorem{StrongCFSGroupDefinition}{Definition}[section]
\theoremstyle{plain}      \newtheorem{FactorSubgroupsLemma}[StrongCFSGroupDefinition]{Lemma}
\theoremstyle{plain}      \newtheorem{StrongCFSGroupSubgroupsLemma}[StrongCFSGroupDefinition]{Lemma}
\theoremstyle{plain}      \newtheorem{StrongCFSGroupTheorem}[StrongCFSGroupDefinition]{Theorem}
\theoremstyle{plain}      \newtheorem{StrongCFSGroupCorollary}[StrongCFSGroupDefinition]{Corollary}
\UseRawInputEncoding
\begin{document}

\title[Groups in which every Lagrange subset is a factor]
{Groups in which every Lagrange subset is a factor}

\author{M.H. Hooshmand and Stefan Kohl}

\address{Department of Mathematics, Shi.C., Islamic Azad University, Shiraz, Iran}

\email{hadi.hooshmand@gmail.com}

\address{Fridtjof-Nansen-Strasse 3, 18106 Rostock, Germany}

\email{stefan.kohl.rostock@gmail.com}


\date{}

\begin{abstract}
  We determine the finite groups \(G\) in which every subset \(A \subseteq G\)
  of cardinality dividing the order of \(G\) is a \emph{factor}, i.e. has a complement
  \(B \subseteq G\) of cardinality \(|G|/|A|\) such that \(G = A \cdot B\)
  or \(G = B \cdot A\).
\end{abstract}

\maketitle

\enlargethispage*{4mm}
\thispagestyle{empty}

\section{Introduction and preliminaries} \label{Introduction}

Given two subsets \(A\) and \(B\) of a group \(G\), we put \(AB := \{ab \ | \ a \in A, b \in B\}\).
We say that the (subset) product \(AB\) is \emph{direct}, and denote it by \(A \cdot B\), if for
every element \(x \in AB\) there is only one pair \((a,b) \in A \times B\) such that \(x = ab\).
Hence we have \(G = A \cdot B\) if and only if \(G = AB\) and the product \(AB\) is direct.
If \(G\) is finite, we have \(G = A \cdot B\) if and only if \(G = AB\) and \(|G| = |A||B|\).

We call \(A\) a \emph{left} (respectively \emph{right}) \emph{factor} of \(G\) if there exists
a subset \(B \subseteq G\) such that \(G = A \cdot B\) (respectively \(G = B \cdot A\)).
In this situation, we call \(B\) a \emph{right} (respectively \emph{left}) \emph{factor complement}
of \(A\). If the subset \(A\) is a left or a right factor of \(G\), we also call it just
a \emph{factor} of \(G\). If it is a left \emph{and} a right factor, we call it a \emph{two-sided}
factor. For example, every subgroup of \(G\) is a left (respectively right) factor with respect to
its right (respectively left) transversal in \(G\), hence it is a two-sided factor of \(G\).

Let \(G\) be a finite group with subsets \(A\) and \(B\) of cardinality \(a\) and \(b\),
respectively. If \(G = A \cdot B\), then \(|G| = ab\). If for the group \(G\) the converse is true
as well, i.e. if for every factorization \(|G| = ab\) of the order of \(G\) into positive integers,
there exist subsets \(A\) and \(B\) of \(G\) such that \(|A| = a\), \(|B| = b\) and
\(G = A \cdot B\), we say that \(G\) has the CFS (`\underline{c}onverse of \underline{f}actorization
by \underline{s}ubsets') property. Problem 20.37 in the Kourovka Notebook~\cite{KourovkaNotebook}
asks whether every finite group has the CFS property (see also~\cite{Hooshmand2014}).
Partial solutions to this problem have been provided in
\cite{BildanovGoryachenkoVasilev2020, Hooshmand2021}. It can be shown that the
CFS property is equivalent to each of the following statements:
\begin{itemize}
  \item for every divisor \(d\) of \(|G|\) there exists a left factor of size \(d\);
  \item for every divisor \(d\) of \(|G|\) there exists a right factor of size \(d\);
  \item for every divisor \(d\) of \(|G|\) there exist a left factor and a right factor
        of size \(d\).
\end{itemize}
But from another perspective, one can imagine a stronger converse of the above proposition.
For the stronger version, we consider the concept of Lagrange subsets of groups (and semigroups)
stated and studied in \cite{Hooshmand2017}. A subset \(A\) of \(G\) is called
a \emph{Lagrange subset} if \(|A|\) divides \(|G|\).
Since \(G = A \cdot B\) and \(G = B \cdot A\) both imply that \(A\) and \(B\) are Lagrange subsets
of \(G\), as the converse, one may ask whether every Lagrange subset is a factor of \(G\).
The answer is negative in general, but we arrive at the next definition.

\begin{StrongCFSGroupDefinition} \label{StrongCFSGroupDefinition}
  We say that a finite group \(G\) has the \emph{strong CFS property} if every Lagrange subset
  is a factor of \(G\).
\end{StrongCFSGroupDefinition}

To better examine such groups, we need the following lemma.

\begin{FactorSubgroupsLemma} \label{FactorSubgroupsLemma}
  Let \(G\) be a group and let \(A \subseteq H \leq G\).
  Then \(A\) is a left, respectively right, factor of \(H\) if and only if it is a left,
  respectively right factor of \(G\). Moreover, the following hold:
\begin{enumerate}[a)]
  \item if \(G = A \cdot B\) for some subset \(B \subseteq G\), then \(H = A \cdot (B \cap H)\), and
  \item if \(H = A \cdot B\) for some subset \(B \subseteq H\), then \(G = A \cdot (BX)\)
        where \(X\) is any right transversal of \(H\) in \(G\).
\end{enumerate}
(Similar results can be obtained for the right case.)
Therefore, \(A\) is a factor (respectively two-sided factor) of \(H\) if and only if
it is a factor (respectively two-sided factor) of \(G\).
\end{FactorSubgroupsLemma}
\begin{proof}
  \begin{enumerate}[a)]

    \item Assume that \(G = A \cdot B\) for some subset \(B \subseteq G\).
          Putting \(C := B \cap H\), we have \(AC \subseteq AH = H\).
          Also, if \(h \in H\), then \(h = ab\) for some \(a \in A\) and \(b \in B\).
          This implies \(h \in AC\) since \(b = a^{-1}h \in A^{-1}H = H\).
          Hence \(H = AC \subseteq AB = A \cdot B\) and so \(H = A \cdot C\).

    \item Assume that \(H = A \cdot B\) for some subset \(B \subseteq H\).
          If \(X\) is a right transversal of \(H\) in \(G\), then
          \(G = H \cdot X = (AB) \cdot X = A \cdot (BX)\).

  \end{enumerate}
\end{proof}

The above lemma has various consequences, one of which is that \(A\) is a factor of \(G\)
if and only if it is a factor of \(\langle A \rangle\). Thus, if \(G\) is finite and \(A\) is not
a Lagrange subset of \(\langle A \rangle\) (i.e., \(|A|\) does not divide \(|\langle A \rangle|\)),
then \(A\) is not a factor of \(G\). The following lemma is another consequence.

\begin{StrongCFSGroupSubgroupsLemma} \label{StrongCFSGroupSubgroupsLemma}
  Every subgroup of a group with the strong CFS property has the strong CFS property as well.
  \end{StrongCFSGroupSubgroupsLemma}
\begin{proof}
  Let \(G\) be a finite group with the strong CFS property, and let \(H \leq G\) be a subgroup.
  If \(A\) is a Lagrange subset of \(H\), then \(A\) is a Lagrange subset
  and so a factor of \(G\). Thus Lemma~\ref{FactorSubgroupsLemma} implies that \(A\)
  is a factor of \(H\), and the proof is complete.
\end{proof}

\section{Main results} \label{Main}

Now, we are ready to characterize all finite groups which have the strong CFS property.

\begin{StrongCFSGroupTheorem} \label{StrongCFSGroupTheorem}
  The finite groups which have the strong CFS property are precisely the following:
  \begin{enumerate}

    \item The trivial group.

    \item The cyclic groups of prime order.

    \item The groups \({\C{2}}^2\), \(\C{4}\), \({\C{2}}^3\) and \({\C{3}}^2\).

  \end{enumerate}
\end{StrongCFSGroupTheorem}
\begin{proof}
  First we check that the listed groups indeed have the strong CFS property.
  Lacking nontrivial divisors of the group order, for the trivial group and for the cyclic groups
  of prime order this is clear. For the other groups \(G\), we check that for every subset
  \(A \subset G\) of cardinality a nontrivial divisor of \(|G|\), there is a set \(B\)
  of cardinality \(|G|/|A|\) such that \(G = AB\). Without loss of generality, we only need
  to consider the sets \(A\) which contain the identity.
  \begin{itemize}

    \item Let \(G = \langle a, b  \ | \ a^2 = b^2 = [a,b] = 1 \rangle \cong {\C{2}}^2\).
          Then we have
          \(G = \{1,a\} \cdot \{1,b\} = \{1,b\} \cdot \{1,a\} = \{1,ab\} \cdot \{1,a\}\).

    \item Let \(G = \langle a \ | \ a^4 = 1 \rangle \cong \C{4}\). Then we have
          \(G = \{1,a\} \cdot \{1,a^2\} = \{1,a^2\} \cdot \{1,a\} = \{1,a^3\} \cdot \{1,a^2\}\).

    \item Let \(G = \langle a, b, c  \ | \ a^2 = b^2 = c^2 = [a,b] = [a,c] = [b,c] = 1 \rangle
          \cong {\C{2}}^3\).
          Putting \(B_1 := \{1,b,c,bc\}\), \(B_2 := \{1,a,c,ac\}\) and \(B_3 := \{1,a,b,ab\}\),
          we have
          \begin{align*}
            G &= \{1,a\} \cdot B_1 = \{1,ab\} \cdot B_1 = \{1,ac\} \cdot B_1 = \{1,abc\} \cdot B_1 \\
              &= \{1,b\} \cdot B_2 = \{1,bc\} \cdot B_2 = \{1,c\} \cdot B_3.
          \end{align*}
          Further,
          \begin{itemize}

            \item if \(A\) is any of \(\{1,b,c,bc\}\), \(\{1,b,c,abc\}\), \(\{1,c,ab,bc\}\),
                  \(\{1,c,ab,abc\}\), \newline \(\{1,b,ac,bc\}\), \(\{1,b,ac,abc\}\),
                  \(\{1,ab,ac,bc\}\) or \(\{1,ab,ac,abc\}\), \newline
                  we can choose \(B = \{1,a\}\),

            \item if \(A\) is any of \(\{1,a,c,ac\}\), \(\{1,a,c,abc\}\),
                  \(\{1,c,ab,ac\}\), \(\{1,a,ac,bc\}\), \newline
                  \(\{1,a,bc,abc\}\) or \(\{1,ab,bc,abc\}\), we can choose \(B = \{1,b\}\),

            \item if \(A\) is any of \(\{1,a,b,ab\}\), \(\{1,a,b,abc\}\),
                  \(\{1,b.ab,ac\}\), \(\{1,a,ab,bc\}\) \newline
                  or \(\{1,ac,bc,abc\}\), we can choose \(B = \{1,c\}\),

            \item if \(A\) is any of \(\{1,b,c,ac\}\), \(\{1,a,c,bc\}\),
                  \(\{1,b,bc,abc\}\) or \(\{1,a,ac,abc\}\), \newline
                  we can choose \(B = \{1,ab\}\),

            \item if \(A\) is any of \(\{1,b,c,ab\}\), \(\{1,c,bc,abc\}\),
                  \(\{1,a,b,bc\}\) or \(\{1,a,ab,abc\}\), \newline
                  we can choose \(B = \{1,ac\}\),

            \item if \(A\) is any of \(\{1,a,c,ab\}\), \(\{1,c,ac,abc\}\),
                  \(\{1,a,b,ac\}\), \(\{1,b,ab,abc\}\), \newline
                  we can choose \(B = \{1,bc\}\), and

            \item if \(A\) is any of \(\{1,a,b,c\}\), \(\{1,c,ac,bc\}\),
                  \(\{1,b,ab,bc\}\) or \(\{1,a,ab,ac\}\), \newline
                  we can choose \(B = \{1,abc\}\).

          \end{itemize}

    \item Let \(G = \langle a, b \ | \ a^3 = b^3 = [a,b] = 1 \rangle \cong {\C{3}}^2\). Then,
          \begin{itemize}

            \item if \(A\) is any of \(\{1,b,b^2\}\), \(\{1,b,ab^2\}\),
                  \(\{1,b,a^2b^2\}\), \(\{1,ab,b^2\}\), \(\{1,ab,ab^2\}\), \(\{1,ab,a^2b^2\}\),
                  \(\{1,b^2,a^2b\}\), \(\{1,a^2b,ab^2\}\) or \(\{1,a^2b,a^2b^2\}\),
                  we can choose \newline \(B = \{1,a,a^2\}\),

            \item if \(A\) is any of \(\{1,a,a^2\}\), \(\{1,a,a^2b\}\),
                  \(\{1,a,a^2b^2\}\), \(\{1,a^2,ab\}\), \(\{1,a^2,ab^2\}\), \newline
                  \(\{1,ab,a^2b\}\) or \(\{1,ab^2,a^2b^2\}\),
                  we can choose \(B = \{1,b,b^2\}\),

            \item if \(A\) is any of \(\{1,a,b\}\), \(\{1,a,ab^2\}\),
                  \(\{1,b,a^2b\}\), \(\{1,a^2,b^2\}\), \(\{1,a^2,a^2b\}\) \newline
                  or \(\{1,b^2,ab^2\}\), we can choose \(B = \{1,ab,a^2b^2\}\), and

            \item if \(A\) is any of \(\{1,a,ab\}\), \(\{1,a,b^2\}\),
                  \(\{1,b,a^2\}\), \(\{1,b,ab\}\), \(\{1,a^2,a^2b^2\}\) \newline
                  or \(\{1,b^2,a^2b^2\}\), we can choose \(B = \{1,ab^2,a^2b\}\).

          \end{itemize}

  \end{itemize}
  Now we prove that there are no other groups which have the strong CFS property. --
  Let \(G\) be a finite group which is not among the groups listed in the theorem.
  We show that there is a subset \(A \subset G\) of cardinality dividing \(n := |G|\) for which
  there exists no subset \(B \subset G\) of cardinality \(n/|A|\) such that \(G = AB\)
  or \(G = BA\). For this, we distinguish the following cases:
  \begin{enumerate}

    \item The group \(G = \langle a \rangle\) is cyclic. As \(G\) is none of the groups listed
          in the theorem, \(n\) has a proper divisor \(d \geq 3\). Now it is easy to check that
          the set \(A := \{1,a^2,a^3, \dots, a^d\}\) of cardinality \(d\) has the desired property
          (try tiling \(G\) with translates \(Aa^k, k \in \ZZ\) of \(A\) -- the key is the ``hole''
          \(a \notin A\) and having at least two consecutive powers \(a^2,a^3, \dots, a^d\)).

    \item The group \(G = \langle a, b \rangle\) is 2-generated. If we can show that there is
          a subset \(\tilde{A} \subsetneq G\) such that \(|\tilde{A}|-1\) divides \(n\), such that
          \(1 \in \tilde{A}\) and such that for every \(g \in G\) we have
          \begin{itemize}
            \item \(|\tilde{A} \cap \tilde{A}g| > 2\) or
            \item \(1 \notin \tilde{A} \cap \tilde{A}g\),
          \end{itemize}
          we can put \(A := \tilde{A} \setminus \{1\}\), and we are done (the key is the ``hole''
          \(1 \notin A\), which cannot be covered by translates \(Ag\) without overlap with \(A\)).
          Now if \(\tilde{A}\)
          \begin{itemize}
            \item contains the ball of radius 2 about 1 with respect to the generating set
                  \(\{a,b\}\),
            \item corresponds to a connected subgraph of the Cayley graph of \(G\) with respect
                  to that generating set, and
            \item has cardinality one more than a divisor of \(n\),
          \end{itemize}
          then \(\tilde{A}\) fulfills our conditions.
          Since the ball of radius 2 about 1 in the free group of rank 2 has cardinality 17,
          the ball of radius 2 about 1 in \(G\) has cardinality at most 17.
          It follows that if \(n\) has a proper divisor \(d \geq 16\), we are done --
          i.e. we can restrict our further considerations to groups of order 6, 8, 9, 10, 12, 14,
          15, 16, 18, 20, 21, 22, 24, 25, 26, 27, 28, 30, 33, 35, 39, 45, 49, 55, 65, 77, 91,
          121, 143, and 169. Since the ball of radius 2 about 1 in the free abelian group of
          rank~2 has cardinality 13, for abelian groups the bound on \(d\) can be lowered to 12.
          In particular, this rules out the group orders 45, 65, 91, 143, and 169.
          Further, by (1.) and Lemma~\ref{StrongCFSGroupSubgroupsLemma}, we can rule out the
          groups which have cyclic subgroups of composite order \(\geq 6\). This leaves us with
          the following 21 groups to be checked, which in the sequel we will deal with one-by-one:
          \(\Sym3\), \(\C{4} \times \C{2}\), \({\rm D}_4\), \({\rm Q}_8\),
          \({\rm D}_5\), \(\Alt{4}\), \({\rm D}_7\), \({\C{4}}^2\),
          \((\C{4} \times \C{2}) \rtimes \C{2}\), \(\C{4} \rtimes \C{4}\),
          \(\C{5} \rtimes \C{4}\), \(\C{7} \rtimes \C{3}\), \({\rm D}_{11}\), \(\Sym{4}\),
          \({\C{5}}^2\), \({\rm D}_{13}\), \({\C{3}}^2 \rtimes \C{3}\),
          \(\C{13} \rtimes \C{3}\), \({\C{7}}^2\), \(\C{11} \rtimes \C{5}\), and \({\C{11}}^2\).
          \begin{itemize}

            \item Assume \(G = \langle (1,2,3), (1,2) \rangle \cong \Sym3\). Then, the set
                  \(A := \{(),(1,2,3)\}\) is not a factor, since for \(B \subset \Sym3\)
                  of cardinality 3, the products \(AB\) and \(BA\) cannot contain all 3 even
                  and all 3 odd permutations. By Lemma~\ref{StrongCFSGroupSubgroupsLemma},
                  this rules out \(G = \Sym4\) as well.

            \item Assume \(G = \langle a, b \ | \ a^4 = b^2 = [a,b] = 1 \rangle
                  \cong \C{4} \times \C{2}\). Then, the set \(A := \{1,a,a^2,b\}\) is not a factor,
                  since there is no \(g \in G\) such that \(Ag = G \setminus A\).
                  By Lemma~\ref{StrongCFSGroupSubgroupsLemma}, this rules out
                  \(G \cong {\C{4}}^2\), \(G \cong (\C{4} \times \C{2}) \rtimes \C{2}\) and
                  \(G \cong \C{4} \rtimes \C{4}\) as well.

            \item Assume \(G = \langle a, b \ | \ a^4 = b^2 = 1, a^b = a^{-1} \rangle
                  \cong {\rm D}_4\). Then, the set \(A := \{1, a, b, a^2b\}\) is not a factor,
                  since there is no \(g \in G\) such that \(Ag = G \setminus A\) or
                  \(gA = G \setminus A\).

            \item Assume \(G = {\rm Q}_8 = \langle i, j \rangle\).
                  Then, the set \(A := \{1,-1,i,j\}\) is not a factor,
                  since there is no \(g \in G\) such that \(Ag = G \setminus A\) or
                  \(gA = G \setminus A\).

            \item Assume \(G = \langle a, b \ | \ a^5 = b^2 = 1, a^b = a^{-1}  \rangle
                  \cong {\rm D}_5\). Then, the set \(A := \{1, b, ab, a^2, a^4\}\) is not a factor,
                  since there is no \(g \in G\) such that \(Ag = G \setminus A\) or
                  \(gA = G \setminus A\). By Lemma~\ref{StrongCFSGroupSubgroupsLemma},
                  this rules out \(G \cong \C{5} \rtimes \C{4}\) as well.

            \item Assume \(G = \langle (1,2,3), (2,3,4) \rangle \cong \Alt4\). Then, the set
                  \[
                    A := \{(), (1,2,3), (1,3,4), (2,3,4), (2,4,3), (1,2)(3,4)\}
                  \]
                  is not a factor, since there is no \(g \in G\) such that \(Ag = G \setminus A\)
                  or \(gA = G \setminus A\).

            \item Assume \(G = \langle a, b \ | \ a^7 = b^2 = 1, a^b = a^{-1}  \rangle
                  \cong {\rm D}_7\). Then, the set \(A := \{1, b, ab, a^2, a^4, a^5, a^6\}\) is not
                  a factor, since there is no \(g \in G\) such that \(Ag = G \setminus A\) or
                  \(gA = G \setminus A\).

            \item Assume \(G = \langle a, b \ | \ a^7 = b^3 = 1, a^b = a^2 \rangle
                  \cong \C{7} \rtimes \C{3}\). Then, the set
                  \(A := \{a, b, a^{-1}, b^{-1}, ab, a^{-1}b, ab^{-1}\}\) is not a factor,
                  since the ``hole'' \(1 \notin A\) cannot be covered by translates \(Ag\) or
                  \(gA\) \((g \in G)\) of \(A\) without overlap (i.e. \(Ag \cap A = \emptyset\)
                  or \(gA \cap A = \emptyset\)).

            \item Assume \(G = \langle a, b \ | \ a^{11} = b^2 = 1, a^b = a^{-1}  \rangle
                  \cong {\rm D}_{11}\).
                  Then, the set \(A := \{1, b, ab, a^2, a^k (4 \leq k \leq 10)\}\) is not
                  a factor, since there is no \(g \in G\) such that \(Ag = G \setminus A\) or
                  \(gA = G \setminus A\).

            \item Assume \(G = \langle a, b \ | \ a^5 = b^5 = [a,b] = 1 \rangle \cong {\C{5}}^2\).
                  Then, the set \(A := \{1, a, a^2, b, a^2b\}\) is not a factor, since translates
                  \(Ag \ (g \in G)\) of that ``U-shape'' cannot tile a \(5 \times 5\) grid (where
                  we identify opposite sides with each other, i.e. things ``wrap around'') without
                  overlap.

            \item Assume \(G = \langle a, b \ | \ a^{13} = b^2 = 1, a^b = a^{-1}  \rangle
                  \cong {\rm D}_{13}\). Then, the set
                  \(A := \{1, b, ab, a^2, a^k (4 \leq k \leq 12)\}\)
                  is not a factor, since there is no \(g \in G\) such that \(Ag = G \setminus A\)
                  or \(gA = G \setminus A\).

            \item Assume \(G = \langle a, b \ | \ a^3 = b^3 = (ab)^3 = (a^2b)^3 = (ab^2)^3 = 1 \rangle
                  \cong {\C{3}}^2 \rtimes \C{3}\). Then, the set
                  \(A := \{a, b, a^{-1}, b^{-1}, ab, a^{-1}b, ab^{-1}, a^{-1}b^{-1}, ba^{-1}\}\)
                  is not a factor, since all intersections \(A \cap Ag\) and
                  \(A \cap gA \ (g \in G)\) are non-empty.

            \item Assume \(G = \langle a, b \ | \ a^{13} = b^3 = 1, a^b = a^3 \rangle
                  \cong \C{13} \rtimes \C{3}\).
                  Then, the set \(A := \{a, b, a^{-1}, b^{-1}, ab, a^{-1}b, ab^{-1}, a^{-1}b^{-1},
                  ba, ba^{-1}, b^{-1}a, b^{-1}a^{-1}, a^2\}\) is not a factor,
                  since all intersections \(A \cap Ag\) and \(A \cap gA \ (g \in G)\) are non-empty.

            \item Assume \(G = \langle a, b \ | \ a^7 = b^7 = [a,b] = 1 \rangle \cong {\C{7}}^2\).
                  Then, the set \(A := \{a, b, a^{-1}, b^{-1}, ab, a^{-1}b, ab^{-1}\}\) is not a
                  factor, since the ``hole'' \(1 \notin A\) cannot be covered by translates \(Ag\)
                  or \(gA\) \((g \in G)\) of \(A\) without overlap.

            \item Assume \(G = \langle a, b \ | \ a^{11} = b^5 = 1, a^b = a^4 \rangle
                  \cong \C{11} \rtimes \C{5}\).
                  Then, the set \(A := \{a, b, a^{-1}, b^{-1}, ab, a^{-1}b, ab^{-1}, a^{-1}b^{-1},
                  ba, ba^{-1}, b^{-1}a\}\) is not a factor, since the ``hole'' \(1 \notin A\) cannot
                  be covered by translates \(Ag\) or \(gA\) \((g \in G)\) of \(A\) without overlap.

            \item Assume \(G = \langle a, b \ | \ a^{11} = b^{11} = [a,b] = 1 \rangle
                  \cong {\C{11}}^2\).
                  Then, the set \(A := \{a, b, a^{-1}, b^{-1}, ab, a^{-1}b, ab^{-1}, a^{-1}b^{-1},
                  a^2, a^{-2}, b^2\}\) is not a factor, since the ``hole'' \(1 \notin A\) cannot be
                  covered by translates \(Ag\) or \(gA\) \((g \in G)\) of \(A\) without overlap.

          \end{itemize}

    \item The group \(G\) is neither cyclic nor 2-generated.
          If \(G\) has a cyclic subgroup of composite order \(\neq 4\), or if \(G\) has
          a 2-generated subgroup other than \({\C{2}}^2\) and \({\C{3}}^2\),
          we are done by Lemma~\ref{StrongCFSGroupSubgroupsLemma}.
          Therefore, if we can rule out \(G = {\C{2}}^4\) and \(G = {\C{3}}^3\), we are done.
          \begin{itemize}

            \item Assume \(G = \langle a, b, c, d \ | \ a^2 = b^2 = c^2 = d^2 = [a,b] = [a,c]
                  = [a,d] = [b,c] = [b,d] = [c,d] = 1 \rangle \cong {\C{2}}^4\).
                  Then, \(A := \{a, b, c, d, ab, ac, ad, bd\}\) is not a factor, since there is
                  no \(g \in G\) such that \(Ag = G \setminus A\) or \(gA = G \setminus A\).

            \item Assume \(G = \langle a, b, c \ | \ a^3 = b^3 = c^3 = [a,b] = [a,c]
                  = [b,c] = 1 \rangle \cong {\C{3}}^3\).
                  Then, the set \(A := \{a, b, c, a^2, b^2, c^2, ab, ac, bc\}\) is not
                  a factor, since all intersections \(A \cap Ag\) and \(A \cap gA \ (g \in G)\)
                  are non-empty.

          \end{itemize}

  \end{enumerate}
\end{proof}

\begin{StrongCFSGroupCorollary} \label{StrongCFSGroupCorollary}
  If \(G\) is a finite group, then the following statements are equivalent:
  \begin{itemize}
    \item \(G\) has the strong CFS property;
    \item every Lagrange subset of \(G\) is a left factor;
    \item every Lagrange subset of \(G\) is a right factor;
    \item every Lagrange subset of \(G\) is a two-sided factor;
    \item for every Lagrange subset \(A\) there exists \(B \subseteq G\)
          such that \(G = A \cdot B = B \cdot A\);
    \item \(G\) is one of the groups mentioned in Theorem~\ref{StrongCFSGroupTheorem}.
  \end{itemize}
  Therefore, in particular no nonabelian group satisfies any of the above equivalent conditions.
\end{StrongCFSGroupCorollary}

\bibliographystyle{amsplain}

\begin{thebibliography}{99}


\bibitem{BildanovGoryachenkoVasilev2020}
  R. Bildanov, V. Goryachenko, and A. V. Vasilev,
  \textit{Factoring nonabelian finite groups into two subsets},
  Sib. Elektron. Mat. Izv., \textbf{17} (2020), 683--689.

\bibitem{Hooshmand2021}
  M. H. Hooshmand,
  \textit{Basic results on an unsolved problem about factorization of finite groups},
  Commun. Algebra, \textbf{49} (7) (2021), 2927--2933.

\bibitem{Hooshmand2017}
  M. H. Hooshmand,
  \textit{Some generalizations of Lagrange theorem and factor subsets for semigroups},
  J. Math. Ext., \textbf{11} (3) (2017), 77--86.

\bibitem{Hooshmand2014}
  \texttt{https://mathoverflow.net/questions/155986/factor-subsets-of-a-finite-group}

\bibitem{KourovkaNotebook}
  E. I. Khukhro and V. D. Mazurov,
  \textit{Unsolved Problems in Group Theory: The Kourovka Notebook},
  no. 19 (resp. no. 20), Sobolev Institute of Mathematics, 2018 (resp. 2022).

\end{thebibliography}

\end{document}